\documentclass{article}
\usepackage{graphicx} 
\usepackage{color}
\usepackage{amsfonts,amsmath,amssymb,amsthm}
\usepackage{verbatim,enumerate}
\usepackage{tikz}

\newtheorem{theorem}{Theorem}
\newtheorem{definition}[theorem]{Definition}
\newtheorem{lemma}[theorem]{Lemma}
\newtheorem{proposition}[theorem]{Proposition}

\newtheorem{corollary}[theorem]{Corollary}
\newtheorem{claim}{Claim}[theorem]

\newtheorem{question}[theorem]{Question}

\newcommand{\SIGMA}{\boldsymbol{\Sigma}}
\newcommand{\PI}{\boldsymbol{\Pi}}

\newcommand{\ww}{\omega^\omega}
\newcommand{\dw}{2^\omega}
\newcommand{\FP}{{\mathsf{FP}}}
\newcommand{\PSP}{{\mathsf{PSP}}}

\title{Functional Projection and Complexity Reduction}
\author{Vera Fischer\thanks{University of Vienna; vera.fischer@univie.ac.at}, Yurii Khomskii\thanks{University of Amsterdam and Hamburg University;  y.d.khomskii@uva.nl}, Yipu Li\thanks{
University of Vienna;  yipu.li@univie.ac.at}, Julia Millhouse\thanks{E\"otv\"os Lor\'and University; julia.marie.millhouse@univie.ac.at}}

\begin{document}

\maketitle

\begin{abstract}
    In \cite{Tornquist13}, T\"ornquist proved that if there exists a $\SIGMA^1_2$ maximal almost disjoint (MAD) family, then there exists a $\PI^1_1$ MAD family. Since then, several other ``complexity reduction’’ proofs have been found for other combinatorial objects, such as maximal independent families or ultrafilters. All of these proofs have relied on $\PI^1_1$-uniformization, so they did not easily lift to higher projective levels.

In this paper, we introduce a principle weaker than Uniformization which we call ``Functional Projection’’, and show that this principle is sufficient for most complexity reduction proofs. On the other hand, we prove that in $L[U]$, the canonical inner model for a measurable cardinal, Functional Projection holds for $\PI^1_2$ relations. As an immediate corollary, we obtain that in $L[U]$ there are $\PI^1_2$ witnesses for many combinatorial families.
\end{abstract}


\section{Introduction}
Maximal combinatorial sets of reals, such as \emph{maximal almost disjoint} (MAD) families, \emph{maximal independent families} (MIF) or ultrafilters, have many applications in set theory and topology, and have been studied by mathematicians for a long time from different perspectives. Descriptive set theory investigates the existence of witnesses for such combinatorial families in the projective hierarchy.

For example, by Mathias’ well-known result \cite{MATHIAS79}, there are no analytic MAD families. On the other hand, in G\"odel's Constructible Universe $L$ it is easy to build a $\SIGMA^1_2$ witness to a MAD family, by  recursion  along the $\SIGMA^1_2$-good well-order of the reals. In \cite{MILLER1989}, Miller developed a sophisticated coding technique to produce a $\PI^1_1$ MAD family in $L$. However, it turned out that this technique can be circumvented when T\"ornquist \cite{Tornquist13} provided a direct combinatorial proof in ZFC of the statement:

\begin{center}\emph{If there is a $\SIGMA^1_2$ MAD family then there is a $\PI^1_1$ MAD family.} \end{center}

Inspired by this result, various researchers  found similar proofs for other combinatorial families, specifically for maximal independent families \cite{Brendle19}, Towers \cite{Fischer22}, Ultrafilter Bases \cite{Schilhan22}, Maximal Ideal Independent Families \cite{Millhouse26} and Hausdorff Gaps \cite{Millhouse25} (the same statement for counterexamples to the Perfect Set Property was already well-known, see for instance the proof of Corollary 25.37 in Jech \cite{Jech}). We call results of this type \emph{complexity reduction} theorems. A detailed survey that includes almost all currently known complexity reductions is \cite[Chapter 2]{Millhouse25}.

\bigskip 
A central ingredient in all such complexity reductions is  \emph{K\^{o}ndo’s $\PI^1_1$} \emph{Uniformization} \emph{Theorem}, and as such, they work only on the $\SIGMA^1_2$ and $\PI^1_1$ levels of the projective hierarchy. Several mathematicians (private communication) have asked whether this could be generalized to the $\SIGMA^1_{n+1}$ and $\PI^1_{n}$ levels for larger $n$, for example by eliminating Uniformization from the proofs.

\bigskip While this general question remains open, we do make significant progress by the discovery that in most   proofs, \emph{Uniformization} can be replaced by a weaker principle that we call \emph{Functional Projection}. Rather than requiring the uniformizing function to be a subset of the original relation, we require only that there exists a function of the same complexity which projects to the same set as the relation does. This weakened principle turns out to be sufficient for many complexity reduction proofs. Thus we are able to lift many complexity reduction proofs to higher pointclasses. At the same time,  \emph{Functional Projection} is genuinely weaker than \emph{Uniformization}. 

It turns out Kanovei and Lyubetsky recently studied \emph{Functional Projection} in  \cite{Kanovei24} and \cite{Kanovei25} (where it is called \emph{uniform projection}). Among other results they proved the following:

\begin{theorem}[Kanovei and Lyubetsky]
    If $V = L$, then $\boldsymbol{\Pi}^1_n$-Functional Projection  holds for all $n\geq 1$.
\end{theorem}

\begin{theorem}[Kanovei and Lyubetsky]
    Assuming ${\rm Con}({\sf ZFC})$, it is consistent that $\boldsymbol{\Pi}^1_n$-Functional Projection holds for all $n\geq 2$, and it is also consistent that $\boldsymbol{\Pi}^1_n$-Functional Projection  fails for all $n\geq 2$.
\end{theorem}

To our knowledge, this is the only prior investigation of this property. In the current paper, we determine the behavior of \emph{Functional Projection} in $L[U]$, the canonical model with one measurable cardinal. We show that in $L[U]$, functional projection holds for all $\boldsymbol{\Pi}^1_n$ relations for all $n\geq 1$. Combined with the well-known existence of a $\SIGMA^1_3$-good well-order of the reals in $L[U]$, an immediate corollary is that in $L[U]$ there are $\PI^1_2$ witnesses for all of the combinatorial families in question. We also determine its behavior under the axiom of Projective Determinacy. Specifically, we show that under $\mathsf{PD}$,  $\boldsymbol{\Pi}^1_n$-\emph{Functional Projection} holds for odd levels, identically to Uniformization, see Figure \ref{fig1}

\begin{figure}[h]
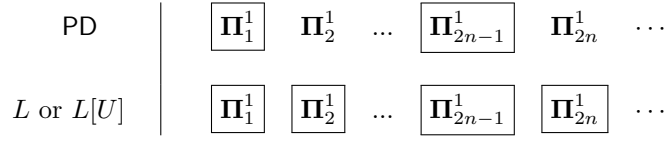
 \label{fig1}
\[\begin{array}{c|cccccc}
{\mathsf{PD}} & \;\;\;\; \;\boxed{\boldsymbol{\Pi}^1_1}    & \boldsymbol{\Pi}^1_2  &  ...   & \boxed{\boldsymbol{\Pi}^1_{2n-1}} & \boldsymbol{\Pi}^1_{2n}     & \cdots\\

\\
{L \text{ or } L[U]} \;\;\; & \;\;\;\;\; \boxed{\boldsymbol{\Pi}^1_1}    & \boxed{\boldsymbol{\Pi}^1_2} &  ...   & \boxed{\boldsymbol{\Pi}^1_{2n-1}} & \boxed{\boldsymbol{\Pi}^1_{2n}}{}     & \cdots
\end{array}\]

\caption{Boxed point-classes satisfy \emph{Functional Projection}.}
\end{figure}

\bigskip This paper is structured as follows: in Section \ref{Functional Projection} we introduce \emph{Functional Projection}, establish some general properties, and study the situation under \textsf{PD}.  In Section \ref{Complexity Reduction}, we prove that \emph{Functional Projection} for $\PI^1_n$ relations leads to complexity reductions for the following combinatorial objects: 

\begin{itemize}
\item Maximal Almost Disjoint families (MAD)
\item Maximal Independent families (MIF)
\item Maximal Ideal Independent Families
\item Maximal Linear Towers and Maximal Towers.
\item Hausdorff Gaps
\end{itemize}

Since most of these proofs are analogous to their original versions, we will only present the first proof (for MAD families) in detail.

In Section \ref{L[U]} we prove that \emph{Functional Projection} holds for $\PI^1_n$ relations in $L[U]$ for all $n$, and we end the paper with some open questions.

\subsection*{Acknowledgments}

We would like to thank the Austrian Science Fund (FWF) for their support through the Emerging Fields Project EFP2661725 (Fischer, Li) and the European Research Council for their support through the Consolidator Grant no. 101231345 (Millhouse).

\subsection*{Declaration}
The research conducted in this work  was {\bf not} assisted by \emph{Large Language Models} or \emph{Generative AI Models} in any way.

\section{Functional Projection}\label{Functional Projection}

Let $X$ be a Polish space and let $\boldsymbol{\Gamma}$ be a class of sets. In most applications $X$ is the Baire or Cantor space, and $\boldsymbol{\Gamma}$  the projective pointclass $\PI^1_n$. Recall the following well-known property:

\begin{definition}  We say that \emph{$\boldsymbol{\Gamma}$-Uniformization} holds if for any $A \subseteq X \times \ww$ of complexity $\boldsymbol{\Gamma}$, there exists a function $F \subseteq A$ with graph  of complexity $\boldsymbol{\Gamma}$,  such that for all $x \in X$
$$\exists y \:(x,y) \in A \;\Leftrightarrow \; \exists y \: (x,y) \in F$$
\end{definition}

K\^{o}ndo's famous uniformization theorem says that $\PI^1_1$-Uniformization and $\SIGMA^1_2$-Uniformization are true. On the other hand, it is easy to see that Uniformization cannot hold both for a  pointclass and its dual pointclass, so $\SIGMA^1_1$-Uniformization and $\PI^1_2$-Uniformization are provably false.  But that is the strongest uniformization result available in ZFC: whether or not $\PI^1_n$-Uniformization and $\SIGMA^1_n$-Uniformization hold for $n\geq 3$ depends on stronger assumptions.

\bigskip

By relaxing the requirement that the function $F$ is a subset of the original set $A$ we obtain a new principle.

\begin{definition} We say that \emph{$\boldsymbol{\Gamma}$-Functional Projection} (abbreviated by $\boldsymbol{\Gamma}$-$\FP$) holds if for any $A \subseteq X \times \ww$ of complexity $\boldsymbol{\Gamma}$, there exists a function $F \subseteq X\times \omega^\omega$   with graph  of complexity $\boldsymbol{\Gamma}$,  such that for all $x \in X$
$$\exists y \:(x,y) \in A \;\Leftrightarrow \; \exists y \: (x,y) \in F$$
\end{definition}

Clearly $\boldsymbol{\Gamma}$-Uniformization implies $\boldsymbol{\Gamma}$-$\FP$, and in Section \ref{L[U]} we will show that the properties are not equivalent. 
We are mainly interested in \emph{Functional Projection}  for  projective classes $\boldsymbol{\Pi}^1_n$. Note that if $\PI^1_n$-$\FP$ holds, then any $\SIGMA^1_{n+1}$ set is the projection of a $\PI^1_n$ function, so this particular consequence of $\PI^1_n$-Uniformization is still preserved.

\begin{center}
\begin{figure}[h]

   \begin{tikzpicture}[scale=1]

\begin{scope}[xshift=0cm]

\draw[thick,->] (0,0) -- (4.2,0) node[below right] {$X$};
\draw[thick,->] (0,0) -- (0,3.3);
\node[left] at (0,3.0) {$\omega^\omega$};

\draw[densely dotted] (0.9,0) -- (0.9,2.15);
\draw[densely dotted] (3.3,0) -- (3.3,2.15);

\draw[thick]
plot[smooth cycle, tension=0.9] coordinates {
    (0.9,1.35)
    (1.1,2.0)
    (1.6,2.35)
    (2.3,2.4)
    (2.9,2.3)
    (3.25,2.0)
    (3.3,1.45)
    (3.0,1.05)
    (2.2,0.95)
    (1.4,1.0)
};

\draw[thick]
plot[smooth] coordinates {
    (0.9,1.55)
    (1.2,1.72)
    (1.7,1.80)
    (2.2,1.72)
    (2.7,1.78)
    (3.1,1.66)
    (3.3,1.55)
};

\node at (2.15,2.7) {$A$};
\node at (2.2,1.9) {$F$};
\node at (2.1,-0.7) {Uniformization};

\end{scope}

\begin{scope}[xshift=6cm]

\draw[thick,->] (0,0) -- (4.2,0) node[below right] {$X$};
\draw[thick,->] (0,0) -- (0,3.3);
\node[left] at (0,3.0) {$\omega^\omega$};

\draw[densely dotted] (0.9,0) -- (0.9,2.25);
\draw[densely dotted] (3.3,0) -- (3.3,2.25);

\draw[thick]
plot[smooth cycle, tension=0.9] coordinates {
    (0.9,1.7)
    (1.1,2.45)
    (1.7,2.8)
    (2.4,2.8)
    (3.0,2.65)
    (3.3,2.2)
    (3.2,1.75)
    (2.7,1.45)
    (1.9,1.4)
    (1.2,1.45)
};

\draw[thick]
plot[smooth] coordinates {
    (0.9,1.0)
    (1.3,1.05)
    (1.8,1.0)
    (2.3,1.08)
    (2.8,1.02)
    (3.3,1.05)
};

\node at (2.15,3.05) {$A$};
\node at (2.1,0.78) {$F$};
\node at (2.1,-0.7) {Functional Projection};

\end{scope}

\end{tikzpicture}
\label{fig2}
\caption{Uniformization vs. Functional Projection}\end{figure}
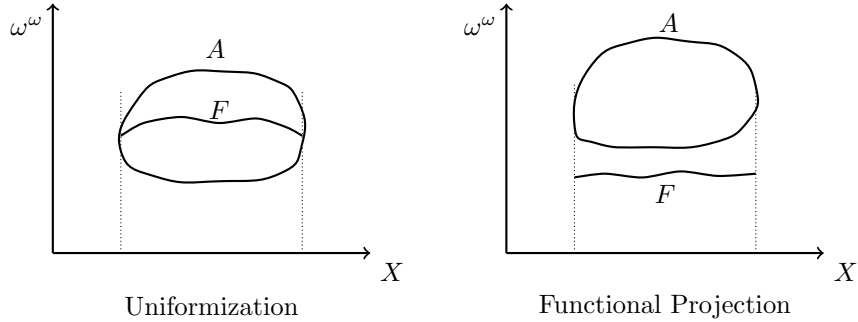
\end{center}

\bigskip Before proceeding we prove some basic facts about \emph{Functional Projection} and Borel injections, which follow from classical results in descriptive set theory.

\begin{definition} Let $X,Y$ be two Polish spaces. We say that a function $f: X \to Y$ is a \emph{Borel function} if $f^{-1}[U]$ is Borel for any open $U \subseteq Y$.\end{definition}

\begin{lemma}[Lusin-Suslin] \label{borelgraph} $f:X \to Y$ is Borel if and only if its graph is Borel in $X\times Y$. \end{lemma}

\begin{proof}
    See the proof of \cite[1G.1 and 2E.4]{Moschovakis2009}.
\end{proof}

\noindent \begin{lemma}[Lusin-Suslin]\label{thm: Luzin-Suslin}  
 If  $f: X \to Y $ is a Borel map, $B\subseteq X$ is a Borel set and $f$ is injective on $B$, then  $f[B]$ is Borel.
\end{lemma}

\begin{proof}
    See \cite[ 1G.5 and 2E.7]{Moschovakis2009}.
\end{proof}

\begin{corollary} \emph{Functional Projection} for closed sets fails. 
\end{corollary}

\begin{proof}
    By Lemma \ref{thm: Luzin-Suslin}, the projection of any closed function $F\subseteq \omega^\omega\times \omega^\omega$ is a Borel set. Let $A \subseteq \ww$ be any analytic set which is not Borel. Then it  cannot be the projection of a closed function. 
\end{proof}

Under determinacy hypothesis, a generalization of the Lusin-Suslin Theorem holds, so we  obtain the behavior of \emph{Functional Projection} under these assumptions.

\begin{lemma}\label{lem: norm}
    Assume $\boldsymbol{\Delta}^1_{2n}$-determinacy. 
    If $f:X\to Y$ is a Borel map, $B\subseteq X$ is a $\boldsymbol{\Delta}^1_{2n+1}$ set, and  $f$ is injective on $B$, then $f[B]$ is $\boldsymbol{\Delta}^1_{2n+1}$.
\end{lemma}

\begin{proof} See, for instance, \cite[6E.7]{Moschovakis2009}. \end{proof}

\begin{proposition}
    Assume  $\boldsymbol{\Delta}^1_{2n}$-determinacy. For all $m\leq n$, $\boldsymbol{\Pi}^1_{2m}$-$\mathsf{FP}$ fails, while $\boldsymbol{\Pi}^1_{2m+1}$-$\mathsf{FP}$ holds.  In particular, under $\mathsf{PD}$, $\boldsymbol{\Pi}^1_{2n}$-$\mathsf{FP}$ fails while $\boldsymbol{\Pi}^1_{2n+1}$-$\mathsf{FP}$ holds for all $n\in \omega$.
\end{proposition}

\begin{proof}
    Let $m\leq n$. For the even case $2m$, by Lemma \ref{lem: norm} the projection of any partial function $F\subseteq \omega^\omega\times \omega^\omega$ with $\boldsymbol{\Delta}^1_{2m+1}$-graph    is   $\boldsymbol{\Delta}^1_{2m+1}$. Let $A \subseteq \ww$ be any $\boldsymbol{\Sigma}^1_{2m+1}$ set which is not $\boldsymbol{\Delta}^1_{2m+1}$. Then it  cannot be the projection of a partial function with $\boldsymbol{\Pi}^1_{2m}$-graph. 

    For the odd case, by the Periodicity Theorem   (see \cite[6C.5]{Moschovakis2009}), it is well-known that  $\boldsymbol{\Pi}^1_{2m+1}$-\emph{Uniformization} holds, and this immediately implies $\boldsymbol{\Pi}^1_{2m+1}$-\emph{Functional Projection}.
\end{proof}

\begin{lemma}\label{lem: Borel-injection}
    $\boldsymbol{\Pi}^1_n$ sets are closed under Borel injections. If $A\subseteq X$ is a $\boldsymbol{\Pi}^1_n$ set and $f:X\to Y$ is a Borel injection, then $f[A]$ is $\boldsymbol{\Pi}^1_n$.
\end{lemma}

\begin{proof}
    We may assume $A\neq X$ as otherwise $f[A]$ is Borel by  Lemma \ref{thm: Luzin-Suslin}. Fix $x_0\not\in A$ and   extend $f^{-1}$ to a function from $Y$ to $X$. $$g(y) = \begin{cases}
        f^{-1}(y) & \text{if }y\in f[X] \\
        x_0 & \text{if } y\not\in f[X]
    \end{cases}$$

   \noindent  By Lemma \ref{thm: Luzin-Suslin}, $f[X]$ is Borel and thus $$Graph(g) = Graph(f^{-1})\cup ((Y - f[X])\times\{x_0\})$$ which is Borel. So $g$ is a Borel function by  Lemma \ref{borelgraph}.
   Since $\PI^1_n$ sets are closed under Borel pre-images, it follows that   $f[A] = g^{-1}[A]$  is $\boldsymbol{\Pi}^1_n$.
\end{proof}

Finally, as a consequence of the above, it is easy to see that  $\PI^1_n$-\emph{Functional Projection}  has the following equivalent formulations: 

\begin{lemma}\label{lem: short-cut-FP}
    For any $n$, the following are equivalent:

    \begin{enumerate}    
    \item $\PI^1_n$-$\FP$.
    \item For any $\PI^1_n$ subset $A \subseteq \ww \times \ww$   there exists a function $F$   with $\PI^1_n$ graph such that:
$$\exists y \:(x,y) \in A \;\Leftrightarrow \; \exists y \: (x,y) \in F$$
    
   \item For every perfect Polish space $X$ and $Y$, for any $\PI^1_n$ subset $A \subseteq X \times Y$    there exists a function $F$   with $\PI^1_n$ graph such that:
$$\exists y \:(x,y) \in A \;\Leftrightarrow \; \exists y \: (x,y) \in F$$ 
    \end{enumerate}
\end{lemma}

\begin{proof}
    This follows by the previous results and the fact that any perfect Polish space is  isomorphic to $\omega^\omega$ via a Borel bijection.
\end{proof}

\bigskip
\section{Complexity Reduction}\label{Complexity Reduction}

In this section, we show that $\PI^1_n$-\emph{Functional Projection} is sufficient to produce the  complexity reduction results that were previously established using $\PI^1_n$-\emph{Uniformization}.

\subsection{Perfect Set Property}

Recall that a subset of the Cantor space has the \emph{Perfect Set Property} ($\PSP$) if it is countable or contains a perfect set. A counterexample to the $\PSP$ can be viewed as similar to a maximal combinatorial object in this sense. 
\begin{theorem}[Folklore]
    Assume $\PI^1_n$-$\FP$. Then all $\SIGMA^1_{n+1}$ sets have the $\PSP$  iff all $\PI^1_n$ sets have the $\PSP$.
\end{theorem}

\begin{proof} For the non-trivial direction, suppose all $\PI^1_n$ sets have the Perfect Set Property. Let $A \subseteq \dw$ be an uncountable  $\SIGMA^1_{n+1}$ set. By $\PI^1_n$-$\FP$, there is  a  function $F: A \to  \ww$ with a $\PI^1_n$ graph and with dom$(F)=A$. Then $F$ (identified with its graph) is also an uncountable set, so by assumption it contains a perfect subset $C \subseteq F$. Then the projection $p[C]$ is  an analytic uncountable subset of $A$, so there is a perfect set $P \subseteq p[C] \subseteq A$ by the Perfect Set Theorem for analytic sets.
\end{proof}

\subsection{Maximal Almost Disjoint Families}

Recall that a \emph{maximal almost disjoint} (MAD) \emph{family} is an infinite collection $\mathcal{A} \subseteq [\omega]^\omega$ such that any two elements of it have finite intersection, and which is maximal with respect to this property. 

\bigskip
The following proof is adapted from T\"ornquist's original result in \cite{Tornquist13}.

\begin{theorem}
    Assume $\PI^1_n$-$\FP$. Then there is a $\boldsymbol{\Sigma}^1_{n+1}$ MAD family iff there is a $\boldsymbol{\Pi}^1_{n}$ MAD family.
\end{theorem}

\begin{proof}
    Let $A$ be a $\boldsymbol{\Sigma}^1_{n+1}$ MAD family. By $\PI^1_n$-$\FP$ there is a partial function $Q\subseteq [\omega]^\omega\times 2^\omega$   with a $\PI^1_n$ graph such that $$x \in A \iff \exists ! w \: (x,w) \in Q$$
    Define the following maps $g_0,g_1: [\omega]^\omega\times 2^\omega \to \mathcal{P}(3\times \omega)$:
    \begin{align*}
        g_0: x,w &\mapsto (\{0\}\times x) \cup (\{2\}\times \{x(n)\mid n\in w\}) \\
        g_1: x,w &\mapsto (\{1\}\times x) \cup (\{2\}\times \{x(n)\mid n\not\in w\})
    \end{align*}

    Let $A' = g_0[Q]\cup g_1[Q]$. To see that $A'$ is almost disjoint, pick $z\neq z'\in A'$, there are two cases:

    \begin{enumerate}
    \item  $z = g_i(x,w), z' = g_i(x',w')$. Without loss of generality assume $i = 0$. We notice that it cannot be that $x = x'$, since if so then, by the fact that $Q$ is a function, $w = w'$ and thus $z = z'$. So $x\neq x'$ and $z\cap z'\subseteq (\{0\}\times (x\cap x')) \cup (\{2\}\times (x\cap x'))$, which is finite since $A$ is a MAD family.

    \item $z = g_i(x,w)$, $z' = g_{1-i}(x',w')$. Without loss of generality assume $i = 0$. If $x = x'$, then $w = w'$ and hence by construction of $g_0$ and $g_1$, $z\cap z' = \emptyset$. If $x\neq x'$, then $z\cap z'\subseteq\{2\}\times x\cap x'$ is finite.

    To see that it is maximal, for $y\in [3\times \omega]^\omega$, $i\in \{0,1,2\}$ write $y_i = \{n\in \omega\mid (i,n)\in y\}$. It suffices to show that for every $y\in [3\times \omega]^\omega$, there is $z\in A'$ s.t. $y\cap z$ is infinite. As $y\in [3\times \omega]^\omega$, one of $y_0,y_1,y_2$ is infinite. If $y_0$ is infinite, then by the fact that $A$ is MAD, there is $x\in A$ s.t. $|x\cap y_0| = \omega$. There is $w$ s.t. $Q(x,w)$ and thus $g_0(x,w)\in A'$ and $y\cap g_0(x,w)\supseteq \{0\}\times (y_0\cap x)$ is infinite. The case for $y_1$ is similar. For the case that $y_2$ is infinite, there is $x\in A$ s.t. $|x\cap y_2| = \omega$ and there is $w$ s.t. $Q(x,w)$. Then $y_2$ intersects one of $\{{x}(n)\mid n\in w\}, \{{x}(n)\mid n\not \in w\}$ infinitely, hence $y$ intersects one of $g_0(x,w),g_1(x,w)$ infinitely. 
    
    \end{enumerate}
    
    This finishes the proof that $A'$ is an infinite MAD family.

    \

    To see that $A'$ is  $\boldsymbol{\Pi}^1_n$, we note that $g_0, g_1$ are Borel injections from $([\omega]^\omega)^2$ to $\mathcal{P}(3\times \omega)$, thus by Lemma \ref{lem: Borel-injection}, $A' = g_0[Q]\cup g_1[Q]$ is $\boldsymbol{\Pi}^1_n$.
\end{proof}

It is worth noticing that based on an observation from \cite{Fischer-Schrittesser21}, with a sufficient absoluteness assumption, the above construction yields a forcing-indestructible MAD family if we start from one.

\begin{definition}
    A forcing poset $\mathbb{P}$ is $\boldsymbol\Sigma^1_{n+1}$-\emph{generic absolute} if for any $\mathbb{P}$-generic filter $G$ over $V$ and any $\boldsymbol\Sigma^1_{n+1}$ property $\varphi(x)$, for any  real $x \in V$: $$\varphi^{V}(x)\iff \varphi^{V[G]}(x)$$
\end{definition}

\begin{corollary}\label{cor: forcing-preservation}
    Assume that $\mathbb{P}$ is $\boldsymbol\Sigma^1_{n+1}$-generic absolute, and assume that $\PI^1_n$-$\FP$  holds. If there is a $\boldsymbol\Sigma^1_{n+1}$ $\mathbb{P}$-indestructible MAD family then there is a $\boldsymbol\Pi^1_{n}$ $\mathbb{P}$-indestructible MAD family.
\end{corollary}

\begin{proof}  Let $A$ be  a $\boldsymbol\Sigma^1_{n+1}$ $\mathbb{P}$-indestructible MAD family. Let $Q\subseteq [\omega]^\omega\times 2^\omega$ be a partial function with $\boldsymbol{\Pi}^1_{n}$-graph  s.t. $$x \in A \iff \exists ! w \: (x,w) \in Q$$
    As before, let $A' = g_1[Q]\cup g_2[Q]$. We need to show that $A'$ is a MAD family not only in $V$ but also in the $\mathbb{P}$-generic extension $V[G]$.

    Let $A^*$ denote the interpretation of $p[Q]$ in $V[G]$;  to be precise $$A^* = \{x \in V[G] \mid \exists y \in V[G] \;(x,y) \in Q^{V[G]}\}.$$
Note that by $\boldsymbol{\Pi}^1_n$-absoluteness it is clear that $A \subseteq A^*$. Moreover,  the statement  ``$p[Q]$ is an a.d. family'' is $$\forall x\neq y (\exists z_1 \exists z_2 ((x,z_1)\in Q\land (y,z_2)\in Q)\to 
    |x\cap y|< \omega)$$
This is a $\boldsymbol{\Pi}^1_{n+1}$ statement that is true in $V$, so by $\boldsymbol\Sigma^1_{n+1}$-absoluteness it is also true in $V[G]$. Hence $A^*$ is an a.d. family in $V[G]$. But $A$ was $\mathbb{P}$-indestructible, so it remains maximal in $V[G]$, so $A = A^*$. 

    Now we show that $Q^V = Q^{V[G]}$. The $\subseteq$-direction is by $\boldsymbol{\Pi}^1_n$-absoluteness. For the $\supseteq$-direction, first note that, by $\boldsymbol{\Pi}^1_{n+1}$-absoluteness, $Q^{V[G]}$ is a graph of a function. If $(x,w)\in Q^{V[G]}$, then $x\in A^*$, and thus $x\in A$. Hence there is $w'\in V$ s.t. $(x,w')\in Q^V$, thus $(x,w')\in Q^{V[G]}$. By functionality of $Q^{V[G]}$, $w = w'$. Thus $(x,w)\in Q^V$.

    Finally, $(g_1[Q]\cup g_2[Q])^{V[G]} = g_1[Q]\cup g_2[Q]$ holds since $Q = Q^{V[G]}$ and Borel functions are absolute on ground model reals.  Thus, the argument in the above theorem that $A'$ is a MAD family can be carried out in $V[G]$. This concludes the proof.
\end{proof}

We mention some consequences of this result:

\begin{corollary}
    If there is a $\boldsymbol{\Sigma}^1_2$ Cohen-indestructible MAD family, then there is a $\boldsymbol{\Pi}^1_1$ Cohen-indestructible MAD family. In particular, there is a $\boldsymbol{\Pi}^1_1$ Cohen-indestructible MAD family in $L$.
\end{corollary}

\begin{proof} Note that every forcing is $\boldsymbol{\Sigma}^1_2$-absolute by Shoenfield absoluteness. For the second claim, construct a $\boldsymbol{\Sigma}^1_2$ Cohen-indestructible MAD family   in $L$ using the good well-order of the reals. 
\end{proof}

\begin{corollary}
    Assume that there is a measurable cardinal. If there is a $\boldsymbol{\Sigma}^1_3$ Cohen-indestructible MAD family, then there is a $\boldsymbol{\Pi}^1_2$ Cohen-indestructible MAD family.
\end{corollary}

\begin{proof}
    By the Martin-Solovay Theorem (see   \cite[Theorem 15.6]{Kanamori}), if there is a measurable cardinal then forcing notions with cardinality less than the measurable cardinal are $\boldsymbol{\Sigma}^1_3$-absolute. Since Cohen forcing is smaller than a measurable cardinal we are done by Corollary \ref{cor: forcing-preservation}.
\end{proof}

In the last section we will use this Corollary to establish that  there is a $\boldsymbol{\Pi}^1_2$ Cohen-indestructible MAD family in $L[U]$.

\subsection{Maximal Independent Families}

Recall the following definition:

\begin{definition}
    A family $A\subseteq [\omega]^\omega$ is \emph{independent} if for all $k,l<\omega$ and distinct $x_0\dots x_k,y_0\dots y_l\in A$, the set $\bigcap_{i\leq k}x_i\setminus \bigcup_{j\leq l}y_j$ is infinite. Such a family is \emph{maximal (MIF)} if it is maximal with respect to inclusion. 
\end{definition}

The following proof is adapted from \cite[Theorem 2.1]{Brendle19}.

\begin{theorem}
    Assume $\PI^1_n$-$\FP$ holds. Then there is a $\boldsymbol\Sigma^1_{n+1}$ MIF iff there is a $\boldsymbol\Pi^1_n$ MIF.
\end{theorem}

\begin{proof}
    Let $A$ be a $\boldsymbol\Sigma^1_{n+1}$ MIF and let $F\subseteq [\omega]^\omega\times 2^\omega$ be a $\boldsymbol\Pi^1_n$ function such that: $$x \in A \iff \exists ! w\in 2^\omega (x,w) \in F.$$  We consider $\omega\cup 2^{<\omega}$ as the disjoint union of $\omega$ and $2^{<\omega}$.  Define \begin{align*}
        g: [\omega]^\omega\times 2^\omega &\to \mathcal{P}(\omega\cup 2^{<\omega})\\
        (x,w)&\mapsto x\cup \{w\upharpoonright n\mid n\in \omega\}
    \end{align*}  
Let $A' = g[F]$. By exactly the same argument as in \cite[Theorem 2.1]{Brendle19}, it follows that  $A'$ is a MIF.

    Finally, $g$ is clearly Borel and injective. Thus $g[F]$ is a $\boldsymbol{\Pi}^1_n$ MIF.
\end{proof}

By an analogous  argument as in Corollary \ref{cor: forcing-preservation}, we obtain the following:

\begin{corollary}
    Let $\mathbb{P}$ be a $\boldsymbol\Sigma^1_{n+1}$-absolute forcing notion and assume that $\PI^1_n$-$\FP$  holds. If there is a $\boldsymbol\Sigma^1_{n+1}$ $\mathbb{P}$-indestructible MIF  then there is a  $\mathbb{P}$-indestructible $\boldsymbol\Pi^1_{n}$ MIF.
\end{corollary}

\subsection{Maximal Ideal Independent Family}

In \cite{Bardyla23, CANCINO21} the notion of a \emph{maximal ideal independent family} was isolated.

\begin{definition}
    An \emph{ideal independent family} is a family $\mathcal{I}\subseteq [\omega]^\omega$ s.t. for all $F \in [\mathcal{I}]^{<\omega}$ and $a\in \mathcal{I}\setminus F$, $a\not\subseteq^* \bigcup F$.
An ideal independent family is \emph{maximal} if it is maximal with respect to inclusion.
\end{definition}

Equivalently, an ideal independent family is maximal iff for every $b\in [\omega]^\omega$ there is a finite $F\subseteq \mathcal{I}$ s.t. either \begin{enumerate}
    \item $b\subseteq^* \bigcup F$ or
    \item there is $a\in \mathcal{I}\setminus F$ s.t. $a\subseteq^* b\cup \bigcup F$.
\end{enumerate}

The cardinal characteristic $\mathbf{s}_{mm}$, defined as the minimum size of an ideal independent family, was studied in detail in \cite{Bardyla23} and \cite{CANCINO21}.

The following proof is adapted from   \cite[Theorem 15]{Millhouse26}.

\begin{theorem}
    Assume $\PI^1_n$-$\FP$. Then there is a $\boldsymbol\Sigma^1_{n+1}$ maximal ideal independent family iff there is a $\boldsymbol\Pi^1_n$ maximal ideal independent family.
\end{theorem}


\begin{proof}
    Let $\mathcal{I}$ be a $\boldsymbol\Sigma^1_{n+1}$ maximal ideal independent family. Let $G$ be a $\boldsymbol\Pi^1_n$ function s.t.  $$x \in \mathcal{I}  \iff \exists ! w\in [\omega]^\omega (x,w) \in G.$$
 We define the following function \begin{align*}
        \chi: [\omega]^\omega \times [\omega]^\omega &\to [\omega]^\omega \\
        (x,w) &\mapsto \{3n\mid n\in x\}\cup \{3n+1\mid n\in w\}
    \end{align*}
 Let $z = \{3n+1,3n+2\mid n\in \omega\}$ and let $\mathcal{I}' = \chi[G]\cup \{z\}$. The same argument as in \cite[Theorem 15]{Millhouse26} shows that $\mathcal{I}'$ is  a maximal ideal independent family. Moreover, the function $\chi$ is clearly a Borel injection so  $\mathcal{I}'$ is  $\boldsymbol{\Pi}^1_n$.
\end{proof}

Again, by an argument similar to  Corollary \ref{cor: forcing-preservation}, we obtain the following:

\begin{corollary}
    Let $\mathbb{P}$ be a $\boldsymbol\Sigma^1_{n+1}$-generic absolute forcing notion. Assume that $\PI^1_n$-$\FP$  holds. If there is a $\boldsymbol\Sigma^1_{n+1}$ $\mathbb{P}$-indestructible maxima ideal independent family,  then there is a $\boldsymbol\Pi^1_{n}$ $\mathbb{P}$-indestructible maxima ideal independent family.
\end{corollary}

\subsection{Towers}

The notion of a \emph{tower} is well-known, for example, from the study of the tower number $\mathfrak{t}$. The idea of looking at towers from a descriptive point of view originates from \cite{Fischer22}.

\begin{definition}
    A collection $A\subseteq [\omega]^\omega$ is said to be a \emph{tower} if it is well-ordered by the relation $\supseteq^*$, where $y\supseteq^* x$ if $x\setminus y$ is finite. If $A$ is only linearly ordered by $\supseteq^*$ we say that it is a \emph{linear tower}. 

    A (linear) tower is maximal if it cannot be extended, i.e. there is no $z\in [\omega]^\omega\setminus A$ s.t. for all $x\in A$, $x\supseteq^* z$. 
\end{definition}

The following proof is adapted from  \cite{Fischer22}:

\begin{theorem}\label{thm: Tower}
Assume $\boldsymbol{\Pi}^1_{n}$-$\mathsf{FP}$. Then there is a $\boldsymbol\Sigma^1_{n+1}$ maximal linear tower iff there is a
$\boldsymbol\Pi^1_n$ maximal linear tower. Moreover, there is a $\boldsymbol\Sigma^1_{n+1}$ maximal tower iff there is a   $\boldsymbol\Pi^1_n$ maximal tower.
\end{theorem}

\begin{proof}
Let $A$ be a $\boldsymbol\Sigma^1_{n+1}$ maximal linear tower, and let
$F\subseteq [\omega]^\omega\times 2^\omega$ be $\boldsymbol\Pi^1_n$ such that
\[
x\in A \iff \exists !w\in 2^\omega\ (x,w)\in F.
\]
For $x\in [\omega]^\omega$, let $x(n)$ be the $n$-th element of $x$.
Define
$g:[\omega]^\omega\times 2^\omega\to [\omega\times\omega]^\omega$ by letting
\[
g(x,w)= (\{0\}\times x)\ \cup\ \left(\bigcup_{n\in\omega}
\{n+1\}\times\bigl(x\setminus x(n+w(n))\bigr)\right).
\]
We show that $g$ is a Borel injection. It is obviously Borel.  For $z\in [\omega\times \omega]^\omega$, let $z_n$ be the $n$th section of $z$. For $(x,w)\neq (x',w')$, may assume $x= x'$ as otherwise $g(x,w)_0\neq g(x',w')_0$. Let $s$ be the largest common initial segment of $w,w'$, may assume $w = s^\frown 0^\frown y$ and $w' = s^\frown 1^\frown y'$, then $$g(x, w)_{|s| + 1} = (x\setminus x(n)) \neq (x\setminus x(n+1)) = g(x,w')_{|s|+1}$$Thus $g$ is injective. 

By an argument analogous to \cite[Theorem 7.2]{Fischer22} it follows that  $g[F]$ is a maximal linear tower is the same as that in \cite{Fischer22}. Also, $g[F]$ is $\boldsymbol{\Pi}^1_n$  by Lemma \ref{lem: Borel-injection}.

\bigskip
The second part of the statement follows by observing that $g$ is order-preserving: for $(x,w),(x',w')\in F$, $$x\subseteq^* x'\iff g(x,w)\subseteq^* g(x',w')$$
Hence, if $A$ is well-ordered by $\subseteq^*$, then  $g[F]$ is as well.
\end{proof}

Unlike previous examples, we cannot easily generalize this proof to  take care of forcing-indestructible towers (i.e., the parallel statement to Corollary \ref{cor: forcing-preservation}). The reason is that in the proof of Corollary \ref{cor: forcing-preservation}, we used that if $A,A^*$ are both MAD families and $A\subseteq A^*$, then $A = A^*$. But this is not the case for maximal towers.

\subsection{Hausdorff Gaps}

The following definition goes back to original work of Hausdorff.
\begin{definition}
    A \emph{pre-gap} is a pair $(A,B)$ where $A,B\subseteq [\omega]^\omega$ are well-ordered by $\subseteq^*$, and for any $x\in A$ and $y\in B$, $x\cap y$ is finite. If $|A| = \kappa$ and $|B| = \lambda$, then $(A,B)$ is called a $(\kappa, \lambda)$-pre-gap.

\bigskip \noindent
    When $(A,B)$ is a $(\omega_1,\omega_1)$-pre-gap, and in addition satisfies the following criterion:

    $$\forall \alpha<\omega_1 \; \forall k<\omega  \; (\{\gamma<\alpha\mid a_\alpha\cap b_\gamma \subseteq k\} \text{ is finite})$$
    it is called a \emph{Hausforff Gap}.\footnote{There is some variation in how the terms ``pre-gap'', ``gap'' and ``Hausdorff Gap'' are defined in the literarure.}
\end{definition}

The following theorem follows from Millhouse's result in \cite{Millhouse25}:

\begin{theorem}
    Assume $\PI^1_n$-$\FP$ and $\PI^1_m$-$\FP$. Then there is a Hausdorff gap $(A,B)$ s.t. $A,B$ are $\boldsymbol{\Sigma}^1_{n+1}, \boldsymbol{\Sigma}^1_{m+1}$ respectively, iff there is a Hausdorff gap $(A',B')$ on $\mathcal{P}(\omega\times \omega)$ s.t. $A',B'$ are $\boldsymbol{\Pi}^1_n$, $\boldsymbol{\Pi}^1_m$, respectively.
\end{theorem}

\begin{proof}
    Let $F,G\subseteq [\omega]^\omega\times 2^\omega$ be $\boldsymbol{\Pi}^1_n$, $\boldsymbol{\Pi}^1_m$ graphs of functions s.t. \[a\in A\iff \exists! w\in 2^\omega ((a,w)\in F) \text{ and }b\in B\iff \exists! v\in 2^\omega ((b,v)\in G)\]
Define $g_A,g_B: [\omega]^\omega \times 2^\omega \to [\omega\times \omega]^\omega$ by letting $$g_A(a,w) = (\{0\}\times a)\cup \left(\bigcup_{n\in\omega}
    \{2n+2\}\times\bigl(a\setminus a(n+w(n))\bigr)\right)$$

    $$g_B(b,v) = (\{0\}\times b)\cup \left(\bigcup_{n\in\omega}
    \{2n+1\}\times\bigl(b\setminus b(n+v(n))\bigr)\right)$$
Let $A' = g_A[F]$, $B' = g_B[G]$. By an argument analogous to \cite[Theorem 2.4.2]{Millhouse25}, it follows that $(A',B')$ forms a Hausdorff gap.  To show that $A',B'$ are respectively $\boldsymbol{\Pi}^1_n$ and $\boldsymbol{\Pi}^1_m$,  it suffices to show that $g_A,g_B$ are Borel injections. This is essentially analogous to the proof of the corresponding statement in  Theorem \ref{thm: Tower}.
\end{proof}








\bigskip
\section{Functional Projection under a Definable Well-order}\label{L[U]}

In this section, we show that \emph{Functional Projection} holds for models with a good projective well ordering. In particular, it holds for all projective levels in $L[U]$, the canonical inner model for a measurable cardinal.   Together with the results in the previous section, this establishes the existence of $\boldsymbol{\Pi}^1_2$ MAD families, maximal independent families etc., in $L[U]$. These are the optimal levels in that model.

We start by recalling the definition of a good projective well-ordering, using the slightly less well-known equivalent formulation of Moschovakis  
\cite[5A]{Moschovakis2009}.

\begin{definition} \label{def: good}
    Let $\leq$ be a well-ordering of a Polish space $X$. We say that $\leq$ is \emph{$\boldsymbol{\Sigma}^1_n$-good} if for any Polish space $Z$, for any $m \geq n$ and any $\boldsymbol{\Sigma}^1_m$-relation $P\subseteq Z\times X$  the following relations are both in $\boldsymbol{\Sigma}^1_m$:
    \begin{align*}
        Q(z,x)\iff \exists y\leq x, P(z,y)\\
        R(z,x)\iff \forall y\leq x, P(z,y)
    \end{align*}
\end{definition}

In words, ``goodness'' means that bounded quantification with respect to $\leq$ preserves the $\boldsymbol{\Sigma}^1_m$ complexity level.

In the literature, a more common definition of a ``$\boldsymbol{\Sigma}^1_n$-good well-order'' states  that $\leq$ is $\boldsymbol{\Sigma}^1_n$ and the relation ``$z$ codes the $\leq$-initial segment before $x$''  is a $\boldsymbol{\Sigma}^1_n$-relation. These definitions are equivalent as shown in the next Lemma.

\begin{lemma}
    Let $\leq$ be a $\boldsymbol{\Delta}^1_n$ well-order of the reals of order type $\omega_1$. Then $\leq$ is $\boldsymbol{\Sigma}^1_n$-good iff there is a relation $R$ on $\ww$ that is $\boldsymbol{\Sigma}^1_n$, and such that $$R(z,x)\iff \{(z)_n\mid n\in \omega\} = \{y\mid y\leq x\}$$ where  $(z)_n$ denotes the map $m\mapsto z(\langle n,m\rangle)$. 
\end{lemma}

\begin{proof} For the forward direction, suppose $\leq$ is $\boldsymbol{\Sigma}^1_n$-good according to our definition. Write  $$R(z,x)\iff \forall n \: ((z)_n\leq x) \land \forall x'\leq x \: \exists n ((z)_n = x').$$

\noindent Conversely, suppose $R$ is given as above. Let  $P\subseteq \ww\times \ww$ be a $\boldsymbol{\Sigma}^1_m$ relation for some $m\geq n$. Then the following equivalences hold: \begin{align*}
        \exists y'\leq y \: P(x,y') &\iff \exists z (R(z,y)\land \exists n P(x,(z)_n))\\
        \forall y'\leq y \: P(x,y') &\iff \exists z (R(z,y)\land \forall n P(x,(z)_n))
    \end{align*}
    Both statements are  $\boldsymbol{\Sigma}^1_m$. \end{proof}

\begin{theorem}\label{thm: FP-in-model}
    For  $n \geq 2$, let $<$ be a $\boldsymbol{\Delta}^1_n$ well-order of the reals, of order-type $\omega_1$, which is   $\boldsymbol{\Sigma}^1_n$-good. Then $\boldsymbol{\Pi}^1_m$-$\mathsf{FP}$ holds for all $m\geq n+1$.
\end{theorem}

\begin{proof} Let $Q\subseteq \omega^\omega\times \omega^\omega$ be a $\boldsymbol{\Pi}^1_m$ relation, for $m \geq n+1$. By Lemma \ref{lem: short-cut-FP} it suffices to find a $\boldsymbol{\Pi}^1_m$ partial function $C\subseteq \omega^\omega\times \omega^\omega$  that has the same projection as $Q$.

    Let $P\subseteq \omega^\omega\times \omega^\omega\times \omega^\omega$ be $\boldsymbol{\Sigma}^1_{m-1}$ such that $$Q(x,w)\iff \forall u\in \ww P(x,w,u)$$
    We define $C$ to be the set of all $(x,\langle w,u\rangle)$ that satisfy all of the following: \begin{enumerate}
        \item $Q(x,w)$
        \item $\forall w'< w \: \exists i\in \omega \: \neg P (x,w',(u)_i)$
        \item $\forall u'< u \:\exists w'< w \: \forall i\in \omega, P(x,w',(u')_i)$
    \end{enumerate}
Notice that Clause 2 is easily seen to be $\boldsymbol{\Pi}^1_{m-1}$ in $x,w,u$, and Clause 3 is $\boldsymbol{\Sigma}^1_{m-1}$ in $x,w,u$ by goodness of $<_M$. It follows that   $C$ is $\boldsymbol{\Pi}^1_m$.

  \bigskip

    Next, we show that $$\exists w\in \ww Q(x,w)\iff\exists w,u\in \ww  \:C(x,\langle w,u\rangle)$$ 

    \noindent $\Leftarrow:$ trivial by Clause 1.

    \noindent $\Rightarrow:$ Let $w$ be  $<$-least such that $Q(x,w)$. Then for every $w'< w$, there is some $u_{w'}$ s.t. $\neg P(x,w',u_{w'})$. As $\{w'\mid w'< w\}$ is countable, there will be a $u$ s.t. $\forall w'< w \: \exists i\in \omega \: \neg P (x,w',(u)_i)$. Fix $u$ to be the $<$-least such element, i.e., $\forall u'< u \: \exists w'<  w \:\forall i\in \omega \: P(x,w',(u')_i)$. Then $C(x,\langle w,u\rangle)$ holds by construction.

    \

    Finally, we show if $C(x,\langle w,u\rangle)$ holds, then $w,u$ have to be unique.  
    Suppose, towards contradiction, that $C(x,\langle w,u\rangle)$ and $C(x,\langle w',u'\rangle)$ hold and wlog. $w< w'$. Then by Clause 1 in the definition of $C(x,\langle w,u\rangle)$, we know that $Q(x,w)$ holds. Yet by Clause 2 in the definition of $C(x,\langle w',u'\rangle)$, there is some $i\in \omega$ such that $\neg P(x,w,(u')_i)$, entailing $\neg Q(x,w)$, a contradiction. Thus $w = w'$. Similarly, assume towards  contradiction that $u< u'$. Then Clause 2 of the definition of $C(x,\langle w,u\rangle)$ and Clause 3 of the definition of $C(x,\langle w,u'\rangle)$ analogously lead to a contradiction.
    
    This completes the proof.
\end{proof}

This is a contrast to the behavior of \emph{Uniformization} under a projective well-order: If there is a $\boldsymbol{\Sigma}^1_n$-good $\boldsymbol{\Delta}^1_n$ well-ordering of the reals, then $\boldsymbol{\Sigma}^1_m$-\emph{Uniformization} holds for all $m\geq n$, and as a consequence $\boldsymbol{\Pi}^1_{m}$-\emph{Uniformization} fails for all $m\geq n$.

\bigskip 
\subsection{Functional Projection in $L$}

First we recall the following result from \cite{Kanovei25}.
\begin{theorem}[Kanovei and Lyubetsky]\label{thm: FP-in-L}
    $L\models \boldsymbol{\Pi}^1_n$-$\mathsf{FP}$ for all $n\geq 1$.
\end{theorem}

 We are interested in an alternative proof of this theorem, which can be generalized more easily to other models, in particular $L[U]$.

\bigskip For  $c\in \ww$, let $E_c$ denote the binary relation coded by $c$ in the following sense:  $$nE_cm\iff c(\langle n,m\rangle) = 0$$
and let $ \mathsf{WF} = \{c \in  \ww \mid  E_c$ is well-founded\}. For $c \in\mathsf{WF}$ we also write $M_c$ to denote the  transitive collapse of $(\omega,E_c)$, i.e., the countable transitive model coded by $c$. More generally, if $M$ is a countable transitive model, we will say ``$c$ codes $M$'' if $M = M_c$.

\begin{proof}[Alternative proof of Theorem \ref{thm: FP-in-L}]
    We work in $L$. The canonical well-order $<_L$ restricted to the reals is a $\boldsymbol{\Delta}^1_2$ relation which is $\boldsymbol{\Sigma}^1_2$-good. Thus by Theorem \ref{thm: FP-in-model}, $L\models \boldsymbol{\Pi}^1_n$-$\mathsf{FP}$ for all $n\geq 3$. K\^{o}ndo's Uniformization solves the case for $n= 1$. Thus we are only left with the case $n = 2$.

    For simplicity of notation, we will prove the result for lightface $\Pi^1_2$, since parameterization is easy.

    Let $R\subseteq \ww\times\ww $ be a ${\Pi}^1_2$ relation. Assume $P\subseteq \ww\times \ww\times \ww$ is ${\Sigma}^1_1$ such that $$R(x,w)\iff \forall u P(x,w,u)$$
    We consider the following relation: $C(x,\langle w,u,c\rangle)$ iff all of the following holds: \begin{enumerate}
        \item $R(x,w)$.
        \item $(c)_0$ codes some $L_\alpha$ and $\alpha$ is the least ordinal s.t. $L_\alpha$ contains $x,w,u$.
        \item $M_{(c)_0}$ satisfies the following: \begin{itemize}
            \item[3.1] $\forall w'<_L w \; \exists i\in \omega \; \neg P(x,w',(u)_i)$.
            \item[3.2] $\forall u'<_L u \; \exists w'<_L w \;\forall i\in \omega \: P(x,w',(u')_i)$
        \end{itemize}
        \item For all $n$, $(c)_{n+1}$ codes some $L_\alpha$ and $\alpha$ is  least  s.t. $L_\alpha$ contains $(c)_n$.
        \item For all $n$, $M_{(c)_{n+1}}$ satisfies: $(c)_n$ is the $<_{L}$-least code of the structure it codes, i.e., $M_{(c)_{n+1}}\models \forall c'<_L (c)_n, M_{c'}\not\cong M_{(c)_n}$.
    \end{enumerate}

    \

    \begin{claim}
        $C$ is ${\Pi}^1_2$
    \end{claim}

    \begin{proof} 
        Clause 1 is $\Pi^1_2$ by assumption, Clause 2 is $\Pi^1_2$ in $x,w,u,c$ as it can be expressed by the conjunction of the two sentences \begin{align*}
         1. & \; (c)_0\in\mathsf{WF} \land (M_{(c)_0}\models \mathsf{ZF}^* + V = L) \land x,w,u\in M_{(c)_0} \\
            2. & \; \forall c' \bigl(( c'\in\mathsf{WF}\land (M_{c'}\models \mathsf{ZF}^* + V = L) \land x,w,u\in M_{c'}) \to M_{(c)_0} \subseteq M_{c'}
           \bigr) 
        \end{align*}
        Here $\mathsf{ZF}^*$ denotes sufficiently much to prove any
        well-founded relation admits a rank function, so that Mostowski absoluteness applies; as well, it is sufficient to prove the absoluteness of $V=L$, in order to guarantee that models that satisfy $\mathsf{ZF}^* +V=L$ are of the form $L_\alpha$. Notice that in case $M_{(c)_0} = L_\alpha$ and $M_{c'} = L_\beta$, $M_{(c)_0} \subseteq M_{c'}$ is equivalent to ``$\exists f\in \ww \text{ bijective}, \forall n,m\in \omega,( nE_{(c)_0}m\leftrightarrow f(n) E_{(c)_0}f(m))$''. Thus the conjunction above is a $\Pi^1_2$-statement since both $(M_{c'}\models \mathsf{ZF}^* + V = L)$ and $(x,w,u\in M_{c'})$ are $\Delta^1_1$ in $x,w,u,c'$.

        Clause 4 is similar. Finally, Clauses 3 and  5 are arithmetical in $x,w,u,c$.
    \end{proof}

    \

    \begin{claim}\label{claim: proj}
    $\exists w R(x,w)\iff \exists ! w \: \exists !u \: \exists !c \: C(x,\langle w,u,c\rangle) $
    \end{claim}

    \begin{proof} The $\Leftarrow$ direction follows immediately by Clause 1 in the definition of $C$.

    For the $\Rightarrow$ direction, assume $\exists w R(x,w)$ and let $w$ be the $<_L$-least element such that $R(x,w)$. Hence for all $w'<_Lw$, there is $z$ s.t.\ $\neg P(x,w',z)$. Since \mbox{$\{w'\in \omega^\omega\mid w'<_Lw\}$} is countable, there is $u$ s.t.\ for all $w'<_Lw$, $\exists n\; \neg P(x,w',(u)_n)$. Let $u$ be the $<_L$-least such element, i.e., if $u'<_Lu$ then there is $w'<_Lw$ s.t. $\forall n P(x,w',(u')_n)$.

        Let $\alpha_0$ be least s.t.\ $x,w,u\in L_{\alpha_0}$. Let $c_0$ be the $<_L$-least element coding $L_{\alpha_0}$. Then recursively pick ${\alpha_{n+1}}$ to be the least ordinal s.t.\  $c_n\in L_{\alpha_{n+1}}$ and $c_{n+1}$ to be the $<_L$-least element coding $L_{\alpha_{n+1}}$. Finally let $c$ be such that $(c)_n = c_n$. We verify that $C(x,\langle w,u,c\rangle)$ holds.

        Clauses 1, 2 and 4 in the definition of $C$ are true by construction. For clause 3 of $C$, consider $L_{\alpha_0}$, the   model coded by $c_0 = (c)_0$.

        For the conjunct 3.1, for arbitrary $w'$, suppose $w'<_Lw$. Then $w'\in L_{\alpha_0}$ and hence by assumption $\exists n \neg P(x,w',(u)_n)$. This is $\Pi^1_1$, and  by Mostowski absoluteness $L_{\alpha_0}\models \exists n \neg P(x,w',(u)_n)$. Hence 3.1 is verified.

        For the conjunct 3.2, for any $u'$ s.t. $u'<_Lu$, again $u'\in L_{\alpha_0}$. We need to find $w'$ s.t. $L_{\alpha_0}\models \forall n  P(x,w',(u')_n)$ and $w'<_Lw$. Since $u'<_{L}u$, by the minimality assumption on $u$, there is $w'<_Lw$ and $\forall n  P(x,w',(u')_n)$. Thus $w'\in L_{\alpha_0}$; moreover by Mostowski absoluteness, $\forall n  P(x,w',(u')_n)$ passes down to $L_{\alpha_0}$. This verifies conjunct 3.2.

        Finally we are left with Clause 5 in the definition of $C$. The argument is similar to the above,  after observing that $M_{c'}\not\cong M_{(c)_n}$ is the following $\Pi^1_1$ statement, thus absolute between $M_{(c)_{n+1}}$ and $L$: $$ \neg \exists f\in \ww \bigl[f \text{ is bijective }\land \forall n,m\in \omega\bigl( n E_{c'} m \leftrightarrow f(n) E_{(c)_n} f(m)\bigr)\bigr]$$

        Now we argue that such $w,u,c$ must be unique, and thus $C$ is functional. Assume $C(x,w,u,c)$ and $C(x,w',u',c')$. Suppose for contradiction that $w\neq w'$, without loss of generality  $w<_Lw'$. Let $L_{\alpha'_0}$ be the model coded by $(c')_0$.  Then $w\in L_{\alpha'_0}$ and $L_{\alpha'_0}\models w<_{L}w'$. But since $R(x,w)$ holds in $L$, by Clause 1 of $C$, and by downward absoluteness for $\Pi^1_2$ predicates, $L_{\alpha'_0} \models R(x,w)$ as well. But this contradicts Clause 3 of $C$, which, together with   $w<_{L}w'$, implies that $L_{\alpha'_0}\models \exists n \neg P(x,w,(u')_n)$ and hence ${\alpha'_0} \models \neg R(x,w)$. This concludes the proof that $w = w'$. The argument that $u = u'$ is a similar minimality argument using clauses 3.1 and 3.2. Finally $c = c'$ is established by inductively proving that $c_n = c'_n$, appealing to Clauses 2, 4 and 5.
    \end{proof}
    \
    This completes the proof.
\end{proof}

\bigskip
\subsection{Facts about $L[U]$}

We now move on to show that \emph{Functional Projection} holds on all levels in $L[U]$. First we need to recall some definitions and standard facts about   $L[U]$. Let  $\kappa$ be a measurable cardinal and $U$  a normal ultrafilter on $\kappa$. We use $V$ to denote a predicate symbol interpreted as a normal ultrafilter in a model of $\mathsf{ZFC}^-$. 

\begin{definition}
    A structure $(M,\in, V)$ is a called a \emph{premouse} if \begin{enumerate}
        \item $(M,\in, V)\models V$ is a normal ultrafilter at $\kappa$.
        \item $M\models \mathsf{ZFC}^-$.
        \item $M = L_\xi[V]$ for some $\xi$.
    \end{enumerate}

    \noindent $(M,\in,V)$ is called \emph{iterable} if there is a system $(M_{\alpha},V_\alpha,j_{\alpha\beta}\mid \alpha\leq\beta\in Ord)$ of transitive well-founded models s.t. \begin{enumerate}
        \item $(M_0,V_0) = (M,V)$,
        \item $j_{\alpha\beta}:M_\alpha\to M_\beta$ are elementary embeddings that commute,
        \item If $\beta = \alpha+1$ is a successor, then $M_\beta= Ult(M_\alpha,V_\alpha)$ and $j_{\alpha\beta}$ is the ultrapower map,
        \item If $\gamma$ is a limit, then $(M_{\gamma}, V_\gamma)$ is the transitive collapse of the direct limit of $((M_\alpha, \in , V_\alpha), j_{\alpha\beta}\mid \alpha\leq \beta <\gamma)$. The maps $j_{\alpha\gamma}$ are given by the direct limit map composed with the collapsing map.
    \end{enumerate}
\end{definition}

The key fact of iterable premice is that they agree with $L[U]$ on initial segments of the constructible order. 

\begin{lemma}\label{lem: <-L[U]}
For any iterable premouse $(M,\in,V)$:
    \begin{enumerate}
        \item For any $x\in L[U]\cap 2^{\omega}$, if $x\in M$ and $y<_{L[U]}x$, then $y\in M$ and  \mbox{$(M,\in,V)\models $``$ y<_{L[V]}x$''.} 
 
        \item If $x,y\in M\cap 2^{\omega}$ then $$y<_{L[U]}x\iff (M,\in,V)\models \text{``}  y<_{L[V]}x\text{''}$$
    \end{enumerate}
\end{lemma}

\begin{proof}
    See, for instance,   Kanamori \cite[p 271]{Kanamori}.
\end{proof}

A consequence of this Lemma is that $<_{L[U]}$ restricted to the reals is a $\Sigma^1_3$-relation.

    \bigskip 
    
    For the sake of the following discussion, we want a real to code a structure $(\omega,E,W)$. For a real $c\in \ww$, write \begin{align*}
    nE_cm\iff c(2\langle n,m\rangle) =0\\
    n\in W_c \iff c(2n+1) = 0
\end{align*}
Then, in case $E_c$ is well-founded, let $\mathcal{M}_c = ({M_c} ,\in ,V)$ be the transitive collapse  of $(\omega, E_c, W_c)$. We say that $c$ codes the structure $\mathcal{M}_c$.

\begin{theorem}\label{thm: complexity in L[U]}
    Assume $V = L[U]$.
    \begin{enumerate}
        \item ``$\mathcal{M}_c$ is an iterable premouse'' is a $\Pi^1_2$-relation on $c$.
        \item $<_{L[U]} \upharpoonright 2^\omega$ is a  \(\Sigma^1_3\)-good \(\Delta^1_3\) well-ordering.
    \end{enumerate}
\end{theorem}

\begin{proof}
    See  \cite[p 272]{Kanamori}.
\end{proof}

Now we define an ordering $<_{ip}$   on iterable premice as follows: $( M,\in, V) <_{ip} ( N,\in, W)$ iff there is $F$ s.t.\ $( L_\xi[F],\in, F\cap L_\xi[F])$ is an iterate of $( M,\in, V)$, $( L_\eta[F],\in, F\cap L_\eta[F])$ is an iterate of $( N,\in, W)$ and $\xi<\eta$. 

\begin{lemma}
    $<_{ip}$ is a well-ordering on iterable  premice.
\end{lemma}

\begin{proof}
    Kanamori \cite[p 273]{Kanamori}.
\end{proof}

\begin{lemma}\label{lem: complexity of ip}
    For $x\in L[U]\cap 2^\omega$, the relation ``$\mathcal{M}_c$ is the $<_{ip}$-least countable iterable premouse containing $x$'' is a $\Pi^1_2$ relation on $c,x$.
\end{lemma}

\begin{proof}
    See the proof of \cite[Theorem 20.21]{Kanamori}.
\end{proof}

\subsection{Functional Projection in $L[U]$}

Now we prove the main result about \emph{Functional Projection} in $L[U]$.

\begin{theorem}\label{thm: FP-in-L[U]}
    $L[U]\models \boldsymbol{\Pi}^1_n$-$\mathsf{FP}$ for all $n\geq 1$.
\end{theorem}

\begin{proof}
    We work in $L[U]$. By Theorem \ref{thm: complexity in L[U]}, the canonical well-order $<_{L[U]}$ restricted to the reals is a $\boldsymbol{\Delta}^1_3$ relation that is $\boldsymbol{\Sigma}^1_3$-good and is of order type $\omega_1^{L[U]}$.  Thus by Theorem \ref{thm: FP-in-model}, $L[U]\models \boldsymbol{\Pi}^1_n$-$\mathsf{FP}$ for all $n\geq 4$.  K\^{o}ndo's Uniformization theorem takes care of the case  $n=1$. So we are  left with the cases $n = 2$ and $n = 3$. For ease of exposition we restrict  attention to lightface point-classes, since the parametrization is easy.
    
    \bigskip

    \noindent \textbf{The case for $\boldsymbol{n = 2}$}. The argument is similar to the proof of Theorem \ref{thm: FP-in-L} with the order $<_{L[U]}$ taking over the role of $<_L$, iterable premice the role of $L_\alpha$, and $<_{ip}$  the role of the initial segment order on $L_\alpha,\alpha<\omega_1$. 
  For the sake of completeness we include a sketch of the proof.

    Let $R\subseteq \ww\times\ww $ be a ${\Pi}^1_2$ relation. Assume $P\subseteq \ww\times \ww\times \ww$ is ${\Sigma}^1_1$ s.t.  $$R(x,w)\iff \forall u P(x,w,u)$$
Define the following relation: $C(x,\langle w,u,c\rangle)$ iff all of the following hold: \begin{enumerate}
        \item $R(x,w)$.
        \item $\mathcal{M}_{(c)_0}$ is the $<_{ip}$-least countable iterable premouse containing $x,w,u$.
        \item Recall that $\mathcal{M}_{(c)_0}$ is the transitive collapse of $(\omega,E_{(c)_0},W_{(c)_0})$. Assume $\mathcal{M}_{(c)_0}$ is $(M_{(c)_0}, \in, V)$, then it satisfies: \begin{itemize}
            \item $\forall w'<_{L[V]}w,\exists i\in \omega \: \neg P(x,w',(u)_i)$.
            \item $\forall u' <_{L[V]} u, \exists w' <_{L[V]} w, \forall i\in \omega \:P(x,w',(u')_i)$.
        \end{itemize}
        \item For all $n$, $\mathcal{M}_{(c)_{n+1}}$ is the $<_{ip}$-least countable iterable premouse containing $(c)_n$.
        \item For all $n$, $\mathcal{M}_{(c)_{n+1}} = ({M_{(c)_{n+1}}}, \in, V_{n+1})$ satisfies $\mathcal{M}_{(c)_{n+1}}\models \forall c' <_{L[V_{n+1}]} (c)_{n}, \  \mathcal{M}_{c'}\not\cong \mathcal{M}_{(c)_n}$.
    \end{enumerate}

    \begin{claim}\label{claim: complexity-of-C}
        $C$ is $\Pi^1_2$.
    \end{claim}

    \begin{proof}
        Clause 1 is $\Pi^1_2$ by assumption. Clauses 3 and 5 are arithmetical. Clauses 2 and 4 are $\Pi^1_2$ by Lemma \ref{lem: complexity of ip}.
    \end{proof}

    The argument that $C$ is a partial function and has the same projection as $R$ is entirely parallel to the proof of Theorem \ref{thm: FP-in-L}.

\bigskip \noindent   \textbf{The case for $\boldsymbol{n = 3}$}. Let $R\subseteq \ww\times\ww$ be the ${\Pi}^1_3$ relation. Assume $$R(x,w)\iff \forall u S(x,w,u) \iff \forall u \exists vT(x,w,u,v)$$
    where $S\subseteq\ww\times \ww\times\ww$ is ${\Sigma}^1_2$ and $T\subseteq \ww\times \ww\times \ww$ is ${\Pi}^1_1$.

    Consider the following relation: $C(x,\langle w,u,v,c\rangle)$ iff all of the following hold: \begin{enumerate}
        \item $R(x,w)$ and $\forall w'<_{L[U]}w \: \exists i\in \omega \:\neg S(x,w',(u)_i)$.
        \item $\mathcal{M}_{(c)_0}$ is the $<_{ip}$-least countable iterable premouse containing $x,w,u,v$.
        \item $\mathcal{M}_{(c)_0} = ({M_{(c)_0}},\in,V)$ satisfies: \begin{itemize}
            \item[3.1] $\forall w'<_{L[V]}w,\exists i\in \omega,\neg S(x,w',(u)_i)$
            \item[3.2] $\forall u'<_{L[V]}u,\exists w'<_{L[V]} w, \forall i ,\exists j, T(x,w',(u')_i, (v)_j)$.
            \item[3.3] $\forall v' <_{L[V]}v ,\exists u'<_{L[V]}u,\forall w'<_{L[V]} w,\exists i, \forall j, \neg T(x,w',(u')_i, (v')_j)$
        \end{itemize}
        \item For all $n$, $\mathcal{M}_{(c)_{n+1}}$ is the $<_{ip}$-least countable iterable premouse containing $(c)_n$.
        \item For all $n$, $\mathcal{M}_{(c)_{n+1}}$ satisfies the statement that ``$(c)_n$ is the least code of the structure it codes, as ordered by its canonical order''.
    \end{enumerate}

    The fact that $C$ is a $\Pi^1_3$ predicate follows the same reasoning as Claim \ref{claim: complexity-of-C}, replacing the sentence in Clause 1 by the following: ``$\forall w'<_{L[U]}w \: \exists i\in \omega \: \neg S(x,w',(u)_i)$'' which is easily seen to be $\Pi^1_3$.

\bigskip
    As before, we now show that $$\exists w R(x,w) \iff \exists ! w \exists ! u \exists ! v \exists ! c \: (C(x,\langle w,u,v,c\rangle)).$$

    For the non-trivial direction, suppose  $\exists w R(x,w)$, let $w$ be the $<_{L[U]}$-least such $w$, let $u$ be  $<_{L[U]}$-least s.t. $\forall w'<_{L[U]}w \:\exists i\in \omega \:\neg S(x,w',(u)_i)$, and let $v$ be  $<_{L[U]}$-least s.t. $\forall u'<_{L[U]}u \:\exists w'<_{L[U]} w \: \forall i  \: \exists j \: T(x,w',(u')_i, (v)_j)$.\footnote{This is analogous to Clauses 3.1, 3.2, with the interpreted order $<_{L[V]}$ replaced by the actual order $<_{L[U]}$.} Let $c_0$ be the $<_{L[U]}$-least element satisfying Clause 2 for the chosen $w,u,v$ and recursively find $c_{n+1}$ satisfying Clause 4 for $c_n$. Let $c$ be such that $(c)_n = c_n$.

    Thus Clauses 1, 2 and 4 will be true by construction. Showing that Clause 5 is true is the  same argument as Theorem \ref{thm: FP-in-L}. It is left to verify Clause 3. Notice all three sub-statements 3.1, 3.2, 3.3 are true in $L[U]$ by construction, with $<_{L[V]}$ replaced by the canonical order $<_{L[U]}$. This implies that  3.1, 3.2 and 3.3 are true in $\mathcal{M}_{(c)_0}$, as all three predicates $\neg S, T,\neg T$, being $\Pi^1_2, \Pi^1_1$ and $\Sigma^1_1$ respectively, are downward absolute by Mostowski absoluteness, and $\mathcal{M}_{(c)_0}$ and $L[U]$ agree on the relevant reals and their order. As an example, we prove clause 3.1 in detail and the other two are similar.
    
  Let $w'\in \mathcal{M}_{(c)_0}$ be such that $w'<_{L[V]}w$, we thus have $w'<_{L[U]}w$. As in the construction, we chose $w$ to be the minimal witness that satisfies $R(x,w)$, and we chose $u$ satisfying $\forall w'<_{L[U]}w \: \exists i\in \omega \: \neg S(x,w',(u)_i)$. Then we have $\exists i\in \omega \: \neg S(x,w',(u)_i)$. This is $\Pi^1_2$ and thus passes down to $\mathcal{M}_{(c)_0}$ by Mostowski absoluteness, showing that $\mathcal{M}_{(c)_0}\models \exists i\in \omega \:\neg S(x,w',(u)_i)$.

    \

    Finally, we show that $w,u,v,c$  are unique. Suppose $C(x,\langle w,u,v,c\rangle)$ and $C(x,\langle w',u',v',c'\rangle)$ hold. 
    Without loss of generality suppose that $w'\leq_{L[U]}w$. If $w'<_{L[U]}w$ then $\exists i \: \neg S(x,w',(u)_i)$ by Clause 1, but we also have $R(x,w')$ by Clause 1 of the definition of $C(x,\langle w',u',v',c'\rangle)$, thus $w' = w$. The minimality argument for $u,v$ works because premice agree on the relevant reals and their order, and the predicates $T,\neg T$ are absolute between them. To see that $c$ is unique, use  the same argument as in Theorem \ref{thm: FP-in-L}.

    This completes the proof.
\end{proof}

By constructing along the $\boldsymbol{\Sigma}^1_3$-good wellorder in $L[U]$, it is easy to obtain $\boldsymbol{\Sigma}^1_3$ combinatorial families in $L[U]$. Combined with the complexity reduction results in the third section, we conclude the following:

\begin{theorem}
    Assume $V = L[U]$. Then there is a $\boldsymbol{\Pi}^1_2$ witness to the following combinatorial families:
    \begin{itemize}
        \item A set without the perfect set property.\footnote{This was known to Silver and Solovay, see  \cite[Theorem 20.21]{Kanamori}.}
        \item A maximal almost disjoint  family (MAD).
        \item A maximal independent family (MIF).
        \item A maximal ideal independent Family.
        \item Maximal linear towers and maximal towers.
        \item Hausdorff Gaps
    \end{itemize}
    Moreover, the MAD family, maximal independent family and maximal ideal independent family can be assumed to be Cohen-indestructible.
\end{theorem}

\section{Open Questions}

In \cite{MILLER1989}, Miller proved that in $L$ there is a $\Pi^1_1$ instance of a \emph{Hamel Basis}, thus it is suspected that the status of a Hamel Basis is similar to the combinatorial families discussed in this article. However, no complexity reduction proofs are known in the literature. Thus the following remains open:

\begin{question}
    Assume $\PI^1_n$-$\FP$. Is it the case that if there is a $\boldsymbol{\Sigma}^1_{n+1}$ Hamel Basis, then there is a $\boldsymbol{\Pi}^1_n$ Hamel Basis?
\end{question}

Of course, the question remains whether we can find complexity reduction proofs without assuming \emph{Functional Projection} or any additional assumptions beyond $\mathsf{ZFC}$. Thus far, the only known instance is Schilhan's result for ultrafilter basis  \cite{Schilhan22}:

\begin{theorem}[Schilhan]
    If there is a $\boldsymbol{\Sigma}^1_{n+1}$ ultrafilter basis then there is a $\boldsymbol{\Pi}^1_n$ ultrafilter basis. 
\end{theorem}

Whether the same holds true for the other combinatorial families discussed in this article remains open:

\begin{question}
    Can the complexity reductions   for the combinatorial families discussed in this article be carried out without assuming  \emph{Functional Projection}?

    Alternatively, is it consistent with $\mathsf{ZFC}$ that there is a  $\boldsymbol{\Sigma}^1_{n+1}$ combinatorial family but no $\boldsymbol{\Pi}^1_n$ combinatorial family, for any or the examples?
    
    If not, what is special about ultrafilter basis?
\end{question}

\bibliography{bib}
\bibliographystyle{abbrv}
\end{document}